\documentclass[reqno]{amsart}
\usepackage{hyperref}

\begin{document}
\title[\hfilneg\hfil fractional Schr\"odinger-Poisson systems]
{Existence and multiplicity results  for the fractional
Schr\"odinger-Poisson systems}

\author[ J. Zhang \hfil \hfilneg]
{Jinguo Zhang}

\address{Jinguo. Zhang \newline
School of Mathematics,
Jiangxi Normal University, Nanchang 330022, China}
\email{jgzhang@jxnu.edu.cn}

\keywords{fractional Schr\"odinger-Poisson systems; fountain theorem;
\hfill\break\indent infinitely many solutions}

\begin{abstract}
 This paper is devoted to study the existence and multiplicity solutions for the nonlinear
 Schr\"odinger-Poisson systems involving fractional Laplacian operator:
\begin{equation}\label{eq*}
 \left\{ \aligned
 &(-\Delta)^{s} u+V(x)u+ \phi u=f(x,u), \quad &\text{in }\mathbb{R}^3,\\
 &(-\Delta)^{t} \phi=u^2, \quad &\text{in }\mathbb{R}^3,\\
  \endaligned
  \right.
\end{equation}
where $(-\Delta)^{\alpha}$ stands for the fractional Laplacian of order $\alpha\in (0\,,\,1)$.
Under certain assumptions on $V$ and $f$, we obtain infinitely many high energy solutions for \eqref{eq*} without
assuming the Ambrosetti-Rabinowitz condition by using the fountain theorem.
\end{abstract}

\maketitle
\numberwithin{equation}{section}
\newtheorem{theorem}{Theorem}[section]
\newtheorem{lemma}{Lemma}[section]
\newtheorem{proposition}{Proposition}[section]
\newtheorem{corollary}{Corollary}[section]
\newtheorem{definition}{Definition}[section]
\newtheorem{remark}{Remark}[section]
\allowdisplaybreaks

\section{Introduction and main results}

In this paper, we are concerned with the existence and multiplicity results for the following nonlinear
Schr\"odinger-Poisson systems involving fractional Laplacian:
\begin{equation}\label{eq1-1}
 \left\{ \aligned
 &(-\Delta)^{s} u+V(x)u+ \phi u=f(x,u), \quad &x\in \mathbb{R}^3,\\
 &(-\Delta)^{t} \phi=u^2, \quad &x\in \mathbb{R}^3,\\
  \endaligned\right.
\end{equation}
where $(-\Delta)^{\alpha}$ is the fractional Laplacian operator for $\alpha=s\,,\,t\in (0\,,\,1)$.
In \eqref{eq1-1}, the first equation is a nonlinear fractional Schr\"odinger equation in which
the potential $\phi$ satisfies a nonlinear fractional Poisson equation. For this reason, system \eqref{eq1-1}
is called a fractional Schr\"odinger-Poisson system, also known as the fractional Schr\"odinger-Maxwell system.

If $\phi=0$ for all $x\in \mathbb{R}^{3}$, system \eqref{eq1-1} reduces to the nonlinear fractional scalar field equation
$$
(-\Delta)^{s}u+V(x)u=f(x,u)\quad x\in \mathbb{R}^{3},
$$
which has been extensively investigated, see for example, \cite{chang2013,dpv2013,fqt2012,s2013,s2014,sz2014}
and references therein. This equation is not only a physically relevant generalization of the classical
NLS but also an important model in the study of fractional quantum mechanics.
In \cite{laskin2000,laskin2002}, Laskin introduced this equation by expanding the Feynman path integral from the
Brownian-like to the L\'evy-like quantum mechanical paths.

To the best of our knowledge, there are only a few article studying the existence and multiplicity of solutions for
nonlinear elliptic system \eqref{eq1-1} involving fractional Laplacian via the variational methods after
it was introduced in \cite{arg2014}.
In \cite{arg2014}, the author studies the following one dimensional system
\begin{equation}\label{eq1-2}
 \left\{ \aligned
 &-\Delta u+ \phi u=a|u|^{p-1}u,\quad &x\in \mathbb{R},\\
 &(-\Delta)^{t} \phi=u^2, \quad &x\in\mathbb{R},\\
  \endaligned\right.
\end{equation}
where $1<p<5$, $0<t<1$. In \eqref{eq1-2} the diffusion is fractional only in the Poisson equation.
Recently, in \cite{jz2015}, the author proved the existence of radial ground state solution of \eqref{eq1-1}
when $V(x)\equiv 0$ and nonlinearity $f(x,u)$ is subcritical or critical growth.
In this paper, we deal with the non-autonomous case when $V(x)$ is not a constant, and
use the Fountain Theorem to find infinitely many large energy solutions  to system \eqref{eq1-1}.
Our system is more general and contain this as a particular case.

Observe that, taking $s=t=1$, the system \eqref{eq1-1} reduces to the classical Schr\"odinger-Poisson system
\begin{equation*}\label{eq1-3}
 \left\{ \aligned
 &-\Delta u+V(x)u+ \phi u=f(x,u), \quad &x\in \mathbb{R}^3,\\
 &-\Delta \phi=u^2, \quad &x\in \mathbb{R}^3,\\
  \endaligned\right.
\end{equation*}
Several papers have dealt with this problem, see, e.g.,
\cite{aa2008,aa2008-1,cv2010,ct2009,lsw2010,r2006,wz2007,z2009,zz2008} and
references therein. In \cite{wz2007}, the authors dealt with the case when $f$ is asymptotically linear at infinity.
 In \cite{aa2008-1,zz2008}, the authors proved the existence of ground state solutions for the case when $f$ is
 superlinear at infinity. Moreover, infinitely many high energy solutions for the superlinear case
  were obtained in \cite{ct2009,lsw2010} via the fountain theorem. In \cite{ct2009}, the following
  Ambrosetti and Rabinowitz condition was assumed,
  \begin{itemize}
  \item [(AR)] There exist $\mu > 4$ and $L>0$ such that
  $$0<\mu F(x,u)\leq u f(x,u),\,\,\,\forall x\in \mathbb{R}^3 ,\,\,\,|u|>L,$$
\end{itemize}
where $F(x,u)=\int_{0}^{u}f(x,\eta )d\eta$. It is well-known that the condition (AR) is crucial in verifying the
boundedness of the $(PS)_{c}$, $c\in \mathbb{R}$, sequence of the corresponding functional.
Without condition (AR), this problem becomes more complicated. In \cite{ct2009},
by using the variant fountain theorem, the authors only considered the case,
where $f(x, u)$ is odd in $u$ and $F(x,u)\geq  0$ for all $x \in \mathbb{R}^3$, $u\in \mathbb{R}$.
The natural question is whether system \eqref{eq1-1} has infinitely many high energy solutions if $f$ is odd but does
not satisfy $F(x, u)\geq 0$. To answer these questions, we assume the following more natural conditions $(f_3)$ or
 $(f_4)$ and give a positive answer. So, we generalize the result in \cite{ct2009},
  and deal with the Schr\"odinger-Poisson with fractional Laplacian operator.
 Moreover. the other main difficulty is to drive the boundedness of the $(PS)_{c}$ sequence of the corresponding functional.
To overcome this difficulty, we will employ the condition $(f_3)$ (or $(f_4)$) to ensure the boundedness of the
$(C)_c$ (or $(PS)_{c}$) sequence. If $f(x,u)$ is odd in $ u$, we obtain infinitely many high energy
solutions by using the fountain theorem.

We introduce the following hypotheses on potential $V$ and the nonlinear term $f$:
\begin{itemize}
\item [(V)] $\inf_{x\in   {\mathbb{R}^3}} V(x) \geq V_0 > 0$, where $V_0$ is a constant.
 Moreover, for every $M>0$,
$\operatorname{meas}(\{x\in {\mathbb{R}}^3:V(x)\leq M\})< \infty$,
 where $\text{meas}(\cdot)$ denote the Lebesgue measure in $\mathbb{R}^3$.
 \item[$(f_1)$]there exists $a_{1}>0$ and $p\in (2\,,\,2^{*}_{s})$ such that
$$|f(x,u)|\leq a_{1}(1+|u|^{p-1}),\quad \forall (x,u)\in \mathbb{R}^{3}\times \mathbb{R},$$
where $2^*_{s}:=\frac{6}{3-2s}$ is the critical exponent in fractional Sobolev inequalities.
\item[($f_2$)] $\lim\limits_{|u|\to \infty}\frac{F(x,u)}{|u|^{4}}=+\infty$ uniformly for $x\in \mathbb{R}^{3}$.
\item[($f_3$)] there exists $L>0$ such that
$$uf(x,u)-4F(x,u)\geq 0,\quad \forall x\in \mathbb{R}^{3},\,\,\,|u|\geq L.$$
\item [($f_4$)] there exits a constant $\theta \geq 1$ such that
$$
\theta \mathcal{F}(x,u) \geq \mathcal{F}(x,\tau u),\quad \forall (x,u)\in \mathbb{R}^3\times \mathbb{R}, \forall \tau \in [0,1],
$$
 where $\mathcal{F}(x,u) = u f(x,u) -4 F(x,u)$.
\item [($f_5$)] $f(x,-u)=-f(x,u)$,  $\forall (x,u)\in {\mathbb{R}^3}\times u\in \mathbb{R}$.
 \end{itemize}
Under the above hypotheses, our results can be stated as follows.

\begin{theorem}\label{th1-1}
Assume that  $(V)$, $(f_1)-(f_3)$ and $(f_5)$ hold. Then when $s\,,\,t\in (0,1)$ satisfying
 $4s+2t\geq 3$ problem \eqref{eq1-1} has infinitely many solutions
 $\{(u_{k}\,,\phi_{u_{k}}^{t})\}$in $H^s({\mathbb{R}}^3) \times D^{t,2}({\mathbb{R}}^3)$
satisfying $I(u_{k})\to +\infty$ as $k\to \infty$, where the functional $I$ is defined in \eqref{eq2-4*}.
\end{theorem}

\begin{theorem}\label{th1-2}
Assume that $(V)$, $(f_1)-(f_2)$ and  $(f_4)-(f_5)$ hold. Then when $s\,,\,t\in (0,1)$
satisfying $4s+2t\geq 3$ problem \eqref{eq1-1} has infinitely many solutions
$\{(u_{k}\,,\phi^{t}_{u_{k}})\}$ in $H^s({\mathbb{R}}^3) \times D^{t,2}({\mathbb{R}}^3)$
satisfying $I(u_{k})\to +\infty$ as $k\to \infty$.
\end{theorem}

\begin{remark}\label{re1-1}\rm
From (AR) condition, for any $x\in \mathbb{R}^{3}$, $|u|\geq L$ and $\eta\in [\frac{L}{|u|}\,,\,1]$, we obtain
$$
\frac{d}{d\eta}\Big(\frac{F(x,\eta u)}{\eta^{\mu}}\Big)=
\frac{f(x,\eta u)\eta u-\mu F(x,\eta u)}{\eta^{\mu+1}}\geq 0,
$$
which implies that
\begin{equation}\label{eq1-4}
F(x,u)\geq \frac{|u|^{\mu}}{L}F(x,\frac{L u}{|u|})\geq \frac{|u|^{\mu}}{L^{\mu}}\inf\limits_{|u|=L}F(x,u),
\end{equation}
for all  $x\in \mathbb{R}^{3}$ and $|u|\geq L$. Since $\mu>4$ and $\inf\limits_{|u|=L}F(x,u)>0$ for all $x\in \mathbb{R}^{3}$,
the inequality \eqref{eq1-4} yields that
$$
\frac{F(x,u)}{|u|^{4}}\geq \frac{|u|^{\mu-4}}{L^{\mu}}\inf\limits_{|u|=L}F(x,u)\to +\infty\,\,,\text{as}\,\,\,|u|\to \infty
$$
and then
$$
uf(x,u)-4F(x,u)\geq (\mu-4)F(x,u)\geq 0
$$
for $|u|$ sufficiently large.
Therefore, (AR) implies $(f_2)$ and $(f_3)$.
\end{remark}

From Theorem \ref{th1-1} and Remark \ref{re1-1}, we get the following Corollary.
\begin{corollary}
The conclusion of Theorem \ref{th1-1} holds, if {\em (V), (AR)},  $(f_1)$ and $(f_5)$ hold.
\end{corollary}

\begin{remark}\label{re1-2}\rm
Condition $(f_4)$, which is weaker than the assumption that:
\begin{itemize}
\item[$(f'_4)$] $\frac{f(x,u)}{u^{3}}$ is increasing in $u>0$ and decreasing in $u<0$,
\end{itemize}
is originally due to Jeanjean \cite{lj1999} for semilinear problem in $\mathbb{R}^{N}$.
\end{remark}
From Theorem \ref{th1-2} and Remark \ref{re1-2}, we get the following Corollary.
\begin{corollary}
The conclusion of Theorem \ref{th1-2} holds, if $(V)$, $(f_1)-(f_2)$, $(f'_4)$ and $(f_5)$ hold.
\end{corollary}

\begin{remark}\label{re1-3} \rm
Obviously, under $(f_2)$ and $(f_3)$ or $(f_2)$ and $(f_4)$ , any $(PS)_{c}$ (or $(C)_{c}$)-sequence of the corresponding energy functional
is bounded, which plays an important role of the application of variational methods.
\end{remark}

The remainder of the paper is arranged as follows.
In Section 2, we present the variational setting for solving our problem.
In Section 3, we give the proofs of the above existence results.

\section{Variational settings and preliminaries}
In this section, we first recall the variational setting for system \eqref{eq1-1}.
A complete introduction to fractional Sobolev spaces can be found
in \cite{Nezza}, we offer below a short review.
We recall that the fractional Sobolev space $W^{\alpha,p}(\mathbb{R}^{N})$
is defined for any $p\in[1,+\infty)$ and $\alpha\in(0,1)$ as
$$
W^{\alpha,p}(\mathbb{R}^{N})
=\big\{u\in L^{p}(\mathbb{R}^{N}): \int_{\mathbb{R}^{N}}
\frac{|u(x)-u(y)|^{p}}{|x-y|^{N+\alpha p}}\,dx\,dy<\infty\big\}.
$$
This space is endowed with the Gagliardo norm
$$
\|u\|_{W^{\alpha,p}}=\Big(\int_{\mathbb{R}^{N}}|u|^{p}\,dx
+\int_{\mathbb{R}^{N}}\frac{|u(x)-u(y)|^{p}}{|x-y|^{N+\alpha p}}\,dx\,dy
\Big)^{\frac{1}{p}}.
$$

If $p=$2, the space $H^{\alpha,2}(\mathbb{R}^{N})$ is denoted by $H^{\alpha}(\mathbb{R}^{N})$,
an equivalent definition of fractional Sobolev spaces is based on Fourier analysis, that is,
$$
H^{\alpha}(\mathbb{R}^{N})=\big\{u\in L^2(\mathbb{R}^{N}):
 \int_{\mathbb{R}^{N}}(1+|\xi|^{2\alpha})|\hat{u}|^2d\xi<\infty\big\},
$$
and the norm can be equivalently written by
$$
\|u\|_{H^{\alpha}}=\Big(\int_{\mathbb{R}^{N}}|\xi|^{2\alpha}|\hat{u}|^2d\xi+
\int_{\mathbb{R}^{N}}|u|^2dx\Big)^{\frac{1}{2}},
$$
where $\hat{u}$ denote the usual Fourier transform of $u$.
Furthermore,  we know that $\|\cdot\|_{H^{\alpha}}$ is equivalent to the norm
$$
\|u\|_{H^{\alpha}}=\Big(\int_{\mathbb{R}^{N}}|(-\Delta)^{\frac{\alpha}{2}}u|^2dx
+\int_{\mathbb{R}^{N}}u^2\,dx\Big)^{\frac{1}{2}}.
$$

In this article, in view of the potential $V(x)$, we consider the subspace
$$
E_{V}=\big\{u\in H^{\alpha}(\mathbb{R}^{N}): \int_{\mathbb{R}^{N}} V(x)u^2\,dx
<\infty\big\}.
$$
Then, $E_{V}$ is a Hilbert space with the inner product
$$
(u,v)_{E_{V}}=\int_{\mathbb{R}^{N}}(|\xi|^{2\alpha}\hat{u}(\xi)\hat{v}(\xi)
+\hat{u}(\xi)\hat{v}(\xi))d\xi+\int_{\mathbb{R}^{N}}V(x)u(x)v(x)\,dx,
$$
 and the norm
$$
\|u\|_{E_{V}}=\Big(\int_{\mathbb{R}^{N}}(|\xi|^{2\alpha}|\hat{u}(\xi)|^2
+|\hat{u}(\xi)|^2)d\xi+ \int_{\mathbb{R}^{N}} V(x)u^2\,dx\Big)^{\frac{1}{2}}.
$$
Furthermore,  we know that $\|\cdot\|_{E_{V}}$ is equivalent to the norm
$$
\|u\|_{E}=\Big(\int_{\mathbb{R}^{N}}(|(-\Delta)^{\frac{\alpha}{2}}u|^2+V(x)u^2)\,dx
\Big)^{\frac{1}{2}},
$$
and the corresponding inner product is
$$
(u,v)_{E}=\int_{\mathbb{R}^{N}}\Big((-\Delta)
^{\frac{\alpha}{2}}u(-\Delta)^{\frac{\alpha}{2}}v+V(x)uv\Big)\,dx.
$$
Throughout out this paper, we  use the norm $\|\cdot\|_{E}$ in $E$.

The homogeneous Sobolev space $D^{\alpha,2}(\mathbb{R}^{3})$ is defined by
$$
D^{\alpha,2}(\mathbb{R}^{3})=\{u\in L^{2^{*}_{\alpha}}(\mathbb{R}^{3}):\, |\xi|^{\alpha}\hat{u}(\xi)\in L^{2}(\mathbb{R}^{3})\},
$$
which is the completion of $C_{0}^{\infty}(\mathbb{R}^{3})$ under the norm
$$
\|u\|_{D^{\alpha,2}}=\Big(\int_{\mathbb{R}^{N}}|(-\Delta)^{\frac{\alpha}{2}}u|^2\,dx\Big)^{\frac{1}{2}}
=\Big(\int_{\mathbb{R}^{3}}|\xi|^{2\alpha}|\hat{u}(\xi)|^{2}d\xi
\Big)^{\frac{1}{2}}.
$$
and the inner product
$$
(u,v)_{D^{\alpha,2}}=\int_{\mathbb{R}^{N}}(-\Delta)
^{\frac{\alpha}{2}}u (-\Delta)^{\frac{\alpha}{2}}v\,dx.
$$

As usual, for $1\leq p< +\infty$, we let
\begin{gather*}
\|u\|_{L^p}=\Big(\int_{\mathbb{R}^{N}}|u(x)|^{p}\,dx\Big)^{\frac{1}{p}},\quad
u\in L^{p}(\mathbb{R}^{N}), \\
\|u\|_{\infty}=\mbox{ess sup}_{x\in \mathbb{R}^{N}} |u(x)|, \quad
u\in L^{\infty}(\mathbb{R}^{N}).
\end{gather*}
To prove our results, the following compactness result is necessary.

\begin{lemma}\label{lem2-1}
$E$ is continuously embedded into $L^{p}(\mathbb{R}^{3})$ for
 $2\leq p \leq 2_{\alpha}^{\ast}:=\frac{6}{3-2\alpha}$ and compactly embedded into
$L^{p}(\mathbb{R}^{3})$ for $2\leq p <2_{\alpha}^{\ast}$.
\end{lemma}

It follows directly from the Lemma \ref{lem2-1} that there are constants
 $C_{p}>0$ such that
$$
\|u\|_{L^p}\leq C_{p}\|u\|_{E}, \quad \forall u\in E, \; p\in [2,2_{\alpha}^{\ast}].
$$

\begin{lemma} \label{lem2-2} For any $\alpha\in (0\,,\,1)$,
$D^{\alpha,2}(\mathbb{R}^{3})$ is continuously embedded into $L^{2^*_{\alpha}}(\mathbb{R}^{3})$,
i.e.,
there exists $S_{\alpha}>0$ such that
$$
\Big(\int_{\mathbb{R}^{3}}|u|^{2^*_{\alpha}}dx\Big)^{\frac{2}{2^*_{\alpha}}}\leq S_{\alpha}
\int_{\mathbb{R}^{3}}|(-\Delta)^{\frac{\alpha}{2}}u|^{2}dx, \quad \forall u\in D^{\alpha,2}(\mathbb{R}^{3}).
$$
\end{lemma}

It is easy to reduce \eqref{eq1-1} to a single equation.
Indeed, using the H\"oleder inequality, for every $u\in E$
\begin{equation}\label{eq2-1}
\aligned
\int_{\mathbb{R}^{3}} u^{2}vdx
&\leq \Big(\int_{\mathbb{R}^{3}}|u|^{\frac{12}{3+2t}}dx\Big)^{\frac{3+2t}{6}}
\Big(\int_{\mathbb{R}^{3}}|v|^{2^*_{t}}dx\Big)^{\frac{1}{2^{*}_{t}}}\\
&\leq S_{t}\|u\|^{2}_{L^{\frac{12}{3+2t}}}\|v\|_{D^{t,2}}\\
&\leq C_{\frac{12}{3+2t}}\,S_{t}\|u\|^{2}_{H^{s}}\|v\|_{D^{t,2}},\\
\endaligned
\end{equation}
where using the following inequality
\begin{equation}\label{eq2-2}
E\hookrightarrow L^{\frac{12}{3+2t}}(\mathbb{R}^{3})\,\,\,
\text{if}\,\,\,2t+4s\geq 3.
\end{equation}
Thus, by the Lax-Milgram theorem, there exists a unique $\phi_u^{t}\in D^{t,2}({\mathbb{R}}^3)$ such that
\begin{equation}\label{eq2-1*}
\int_{\mathbb{R}^{3}}v(-\Delta)^{t}\phi_{u}^{t}dx=\int_{\mathbb{R}^{3}}(-\Delta)^{\frac{t}{2}}\phi_{u}^{t} (-\Delta)^{\frac{t}{2}}vdx=\int_{\mathbb{R}^{3}}u^{2}vdx,\quad v\in D^{t,2}(\mathbb{R}^{3}).
\end{equation}
Therefore, $\phi_{u}^{t}$ satisfies the Poisson equation
$$
(-\Delta)^{t}\phi_{u}^{t}=u^{2},\quad x\in \mathbb{R}^{3},
$$
and we can write an integral expression for $\phi_u^{t}$ in the form:
\begin{equation}\label{eq2-3}
  \phi_u^{t}(x)=c_{t}\int_{{\mathbb{R}}^3}\frac{u^2(y)}{|x-y|^{3-2t}}\,\mathrm{d}y,\quad x\in \mathbb{R}^{3},
\end{equation}
which is called $t$-Riesz potential, where $$
c_{t}=\pi^{-\frac{3}{2}}2^{-2t}\frac{\Gamma(\frac{3}{2}-2t)}{\Gamma(t)}.
$$
It follows from \eqref{eq2-3} that $\phi_{u}^{t}(x)\geq 0$ for all $x\in \mathbb{R}^{3}$.
 Combining \eqref{eq2-1} and \eqref{eq2-1*}, we have
$$
\|\phi_{u}^{t}\|_{D^{t,2}}\leq S_{t}\|u\|^{2}_{L^{\frac{12}{3+2t}}}\leq C_{1}\|u\|_{H^{s}}^{2},\quad \text{if}\,\,2t+4s\geq 3.
$$
Hence, by the H\"older inequality and Lemma \ref{lem2-1}, we get
\begin{equation}\label{eq2-3*}
\aligned
\int_{\mathbb{R}^{3}}\phi_{u}^{t}u^{2}dx
&\leq \Big(\int_{\mathbb{R}^{3}}|\phi_{u}^{t}|^{\frac{1}{2^*_{t}}}dx\Big)^{\frac{1}{2^*_{t}}}
\Big(\int_{\mathbb{R}^{3}}|u^{2}|^{\frac{6}{3+2t}}dx\Big)^{\frac{3+2t}{6}}\\
&\leq \widetilde{C}_{1}\|\phi_{u}^{t}\|_{D^{t,2}}\|u\|^{2}_{H^{s}}\\
&\leq \widetilde{C}_{2}\|u\|^{4}_{H^{s}},\\
\endaligned
\end{equation}
where $\widetilde{C}_{1}$, $\widetilde{C}_{2}>0$.
Substituting \eqref{eq2-3} in to \eqref{eq1-1}, we can rewrite \eqref{eq1-1} in the
following equivalent form
\begin{equation*}\label{eq1-1*}
(-\Delta)^{s}u+V(x)u+ \phi_{u}^{t}u=f(x,u),\quad \quad x\in \mathbb{R}^{3}.
\end{equation*}
We define the energy function $I:\,E\to \mathbb{R}$ by
\begin{equation}\label{eq2-4*}
I(u) = \frac{1}{2}\|u\|_E^2+\frac{1}{4}\int_{{\mathbb{R}}^3}\phi_u^{t}
u^2\,\mathrm{d}x-\int_{{\mathbb{R}}^3}F(x,u)\,\mathrm{d}x.
\end{equation}
From $(f_1)$ and \eqref{eq2-3*}, $I$ is well-defined. Furthermore, it is well-known
that $I$ is $C^1$ functional with derivative given by
\begin{equation}\label{eq2-4}
  \langle I'(u),v\rangle=\int_{{\mathbb{R}}^3}\Big((-\Delta)^{\frac{s}{2}}u
 \cdot (-\Delta)^{\frac{s}{2}} v+V(x)uv+\phi_u^{t} uv-f(x,u)v\Big)\,\mathrm{d}x,\quad \forall v\in E.
\end{equation}
Obviously, it can be proved that if $u$ is a critical points of $I$, then the pair $(u\,,\,\phi_{u}^{t})$ is
 a  solutions of system \eqref{eq1-1}.

For reader's convenience, we introduce the Cerami condition (C),
which was established by Cerami \cite{Cerami1978}.

\begin{definition} \label{def2.4} \rm
Let $(X\,,\,\|\cdot\|)$ be a real Banach space, $\Phi\in C^1(X\,,\,\mathbb{R})$.
We say that $\Phi$ satisfies Cerami condition at level $c\in \mathbb{R}$ ( for short $(C)_{c}$) 
if any sequence $\{u_{n}\}\subset X$ such that $\Phi (u_n)\to c$ and
$(1+ \|u_n\|)\| \Phi' (u_n)\| \to 0$ as $n\to \infty$ has a convergence subsequence.
\end{definition}

In order to prove our main results, we shall use the following  Fountain Theorem.
Let $X$ be a reflexive and separable Banach space with the norm $\|\cdot\|$ and
$X=\overline{\bigoplus_{i\in \mathbb{N}}X_{i}}$
with $\text{dim}X_{i}<\infty$ for all $i\in \mathbb{N}$.
Set
$$
W_k=\bigoplus \limits_{i=1}^k X_i\quad \text{and}\quad
Z_k=\overline{\bigoplus _{i\geq k}X_i}.
$$
\begin{theorem}[Fountain Theorem]\label{ft}
Assume that function $\Phi\in C^{1}(X\,,\,\mathbb{R})$ satisfies $\Phi(-u)=\Phi(u)$. For almost every $k\in \mathbb{N}$,
there exist $\rho_{k}>r_{k}>0$ such that
\begin{itemize}
\item [(i)] $a_k:=\max _{u\in W_k, \|u\|=\rho_k}\Phi(u)\leq 0$,
 \item [(ii)] $b_k:=\inf _{u\in Z_k, \|u\|=r_k}\Phi(u)\to +\infty$,
 as $k\to \infty$,
    \item [(iii)] $\Phi$ satisfies the $(C)_{c}$-condition for all $c>0$.
  \end{itemize}
Then $\Phi$ has a sequence of critical points $\{u_{k}\}$ such that $\Phi(u_{k})\to +\infty$ as $k\to \infty$.
\end{theorem}

\begin{remark}\rm
 In \cite{r1986}, the fountain theorem and mountain pass theorem were established under the $(PS)$ condition
respectively. Since the deformation theorem also holds under the $(C)_c$ condition, these theorems are true when the $(C)_c$
condition is used instead of the $(PS)$ condition.
\end{remark}

\section{Proof of Theorem \ref{th1-2} and \ref{th1-2}}
In this section, we shall apply the Fountain Theorem to find the critical points of $I$.
We first show that the functional $I$ satisfies the $(C)_{c}$ condition for any $c\in \mathbb{R}$.

\begin{lemma} \label{lem3-1}
Suppose that $(V)$, $(f_1)-(f_3)$ hold. Then the functional $I$
satisfies the $(C)_{c}$-condition for all $c\in \mathbb{R}$ .
\end{lemma}

\begin{proof}
Let $\{ u_{n} \}\subset E$ be a  $(C)_{c}$ sequence of $I$, that is,
\begin{equation}\label{eq2-13}
I(u_{n})\to c\,\,\,\text{and}\,\,\,  (1+\| u_{n} \|_{E})I' (u_{n})\to 0\,\,\text{as}\,\,\,n\to \infty.
\end{equation}

In what follows, we shall show that $\{u_{n}\}$ is bounded in $E$.
Otherwise, up to a subsequence,  $\{u_{n}\}$ is unbounded in $E$,
and we may assume that $\|u_{n}\|_{E}\to \infty$ as $n\to \infty$.
We define the sequence $\{w_{n}\}$ by
$$
w_{n}=\frac{u_{n}}{\|u_{n}\|_{E}},\quad n=1\,,\,2\,,\,3\,,\,\cdot\cdot\cdot.
$$
Clearly, $\{w_{n}\}\subset E$ and $\|w_{n}\|_{E}=1$ for all $n\in \mathbb{N}$.
Going over, if necessary, to a subsequence, we may assume that
\begin{equation} \label{eq2-14}
\aligned
   &w_{n}\rightharpoonup w  \quad \text{weakly in }\,\,\,E,\\
   &w_{n}\to w            \quad \text{strongly in }\,\,L^{p}(\mathbb{R}^3), \,\,2 \leq p < 2^{*}_{s},\\
   & w_{n}(x)\to w(x)  \quad \text{a.e.}\,\,x\in \mathbb{R}^3.
\endaligned
\end{equation}

We first consider the case $w\neq 0$ in $\mathbb{R}^{3}$,
then the set $\Omega := \{ x \in \mathbb{R}^3 | w(x) \neq 0 \}$
has positive Lebesgue measure. For all $x\in \Omega$,
we have $|u_{n}(x)|\to \infty$ as $n\to \infty$,
so that, using $(f_2)$, for all $x\in \Omega$,
$$
  \frac{F(x, u_{n})}{\|u_{n}\|_{E}^{4}}=\frac{F(x, u_{n})}{|u_{n}|^{4}}\frac{|u_{n}|^{4}}{\|u_{n}\|_{E}^{4}}
= \frac{F(x, u_{n})}{|u_{n}|^{4}}|w_{n}|^{4} \to +\infty \,\,\,\text{as}\,\,\,n \to \infty,
$$
and then, via Fatou's Lemma,
\begin{equation}\label{eq2-5}
  \int_{\Omega} \frac{F(x, u_{n})}{\|u_{n}\|_{E}^{4}} \,\mathrm{d}x \to +\infty\,\,\,\text{as}\,\,\,n \to \infty.
\end{equation}

On the other hand, by $(f_2)$, there exists $L>0$ such that
\begin{equation}\label{eq2-6}
F(x,u)\geq 0,\quad \forall x\in \mathbb{R}^{3}\,,\,\,|u|>L.
\end{equation}
Moreover, it follows from $(f_1)$ that for any $\varepsilon>0$ there exists $c(\varepsilon)>0$
such that for all $x\in \mathbb{R}^{3}$, $|u|\leq L$, we have
\begin{equation}\label{eq2-6*}
|f(x,u)|\leq \varepsilon|u|+c(\varepsilon)|u|^{p}.
\end{equation}
Then, by the mean value theorem, for all $|u|<L$,  we obtain
\begin{equation}\label{eq2-9}
\aligned
|F(x,u)|
&=|F(x,u)-F(x,0)|= \int_{0}^{1}|f(x,\eta u)ud\eta\\
&\leq \frac{\varepsilon}{2}|u|^{2}+\frac{c(\varepsilon)}{p}|u|^{p}\\
&\leq c_{1}|u|^{2},
\endaligned
\end{equation}
where $c_{1}=\frac{\varepsilon}{2}+c(\varepsilon)\frac{L^{p-2}}{p}>0$. Combining this with \eqref{eq2-6}, we have
\begin{equation}\label{eq2-7}
F(x,u)\geq -c_{1}|u|^{2},\quad \forall (x,u)\in \mathbb{R}^{3}\times \mathbb{R},
\end{equation}
which implies that there exists $c_{2}>0$ such that
$$
F(x,u_{n})\geq -c_{2}|u_{n}|^{2},\quad \text{for all}\,\,\,x\in \mathbb{R}^{3}\setminus \Omega.
$$
Hence, we obtain
$$
\aligned
\int_{\mathbb{R}^{3}\setminus \Omega}\frac{F(x,u_{n})}{\|u_{n}\|^{4}_{E}}dx
&\geq- \frac{c_{2}}{\|u_{n}\|^{4}_{E}}\int_{\mathbb{R}^{3}\setminus \Omega}|u_{n}|^{2}dx\\
&\geq- \frac{c_{2}}{\|u_{n}\|^{4}_{E}}\int_{\mathbb{R}^{3}}|u_{n}|^{2}dx\\
&\geq -c_{3}\frac{\|u_{n}\|^{2}_{E}}{\|u_{n}\|_{E}^{4}},\quad c_{3}>0,
\endaligned
$$
which implies that
\begin{equation}\label{eq2-7}
\liminf\limits_{n\to \infty}\int_{\mathbb{R}^{3}\setminus \Omega}\frac{F(x,u_{n})}{\|u_{n}\|^{4}_{E}}dx\geq0.
\end{equation}
So,  combining \eqref{eq2-5} with \eqref{eq2-7}, one has
\begin{equation}\label{eq2-8}
\lim\limits_{n\to \infty}\int_{{\mathbb{R}}^3}\frac{F(x,u_{n})}{ \|u_{n}\|_{E}^{4} } \,dx=
\lim\limits_{n\to \infty}\Big(\int_{\Omega}+\int_{\mathbb{R}^{3}\setminus \Omega}\Big)\frac{F(x,u_{n})}{ \|u_{n}\|_{E}^{4} } \,dx=+\infty.
\end{equation}

Note
$$
\frac{1}{2}\|u_{n}\|_{E}^{2}+\frac{1}{4}\int_{\mathbb{R}^{3}}\phi_{u_n}^{t}u_{n}^{2}dx-\int_{\mathbb{R}^{3}}F(x,u_{n})=c+o(1).
$$
Dividing both sides by $\|u_{n}\|_{E}^{4}$ and letting $n\to \infty$, we deduce via \eqref{eq2-3*} that
\begin{equation*}\label{3.5}
\aligned
4\int_{{\mathbb{R}}^3}\frac{F(x,u_{n})}{ \|u_{n}\|_{E}^{4} }dx
&= \frac{2}{\|u_{n}\|_{E}^{2}} +\frac{\int_{{\mathbb{R}}^3}\phi_{u_{n}}^{t} u_{n}^2 dx}{\|u_{n}\|_{E}^{4}}
-\frac{4c}{\|u_{n}\|_{E}^{4}} + o(\| u_n \|_{E}^{-4})\\
&\leq \widetilde{C}_{2}+ o(\| u_n \|_{E}^{-4})< \infty,
\endaligned
\end{equation*}
where $\widetilde{C}_2$ is a positive constant. This contradicts \eqref{eq2-8}.

For the second  case $w=0$. It follows from \eqref{eq2-6*}and \eqref{eq2-9}  that,
for all $x\in \mathbb{R}^{3}$ and $|u|\leq L$,
$$
|uf(x,u)-4F(x,u)|\leq (\varepsilon+4a_{1})|u|^{2}+c(\varepsilon)|u|^{p}\leq c_{4}|u|^{2},
$$
where $c_{4}=(\varepsilon +4a_{1})+c(\varepsilon)L^{p-2}>0$. This, together with $(f_{3})$, obtain that
\begin{equation}\label{eq2-12}
uf(x,u)-4F(x,u)\geq -c_{5}|u|^{2},\quad \forall (x,u)\in \mathbb{R}^{3}\times \mathbb{R},
\end{equation}
where $c_5$ is a positive constant.
Therefore, from \eqref{eq2-13},\eqref{eq2-14} and \eqref{eq2-12}, for $n$ large enough, we get
\begin{equation*}\label{eq2-15}
\aligned
c+o(1)
&= I(u_n)-\frac{1}{4}\langle I'(u_n),u_{n}\rangle\\
&= \frac{1}{4}\| u_{n} \|^{2}_{E}
 + \frac{1}{4} \int_{{\mathbb{R}}^3}\Big(f(x,u_n)u_{n}-4F(x, u_{n})\Big)\,dx\\
 &\geq \frac{1}{4}\| u_{n} \|^{2}_{E}-\frac{1}{4}c_{5}\int_{\mathbb{R}^{3}}|u_{n}|^{2}dx\\
 &=\frac{1}{4}\Big(1-c_{5}\int_{\mathbb{R}^{3}}|w_{n}|^{2}dx\Big)\|u_{n}\|_{E}^{2}\to \infty,\\
\endaligned
\end{equation*}
as $n\to \infty$, which is contradiction.

In any case, we deduce a contradiction. Hence $\{u_{n}\}$ is bounded in $E$.

Next, we verify that $\{u_{n}\}$ has a convergent subsequence.
Without loss of generality, we assume that
\begin{equation*}\label{eq2-16}
\aligned
& u_{n}\rightharpoonup u,\,\,\,\text{weakly in}\,\,\,E;\\
& u_{n}\to u,\,\,\,\text{strongly in}\,\,\,L^{p},\,\,2\leq p<2^*_{s}.\\
\endaligned
\end{equation*}
By \eqref{eq2-4}, we easily get
\begin{equation*}\label{eq2-17}
\aligned
\|u_n-u\|_E^2& =\langle I'(u_n) - I'(u), u_n-u\rangle+
  \int_{{\mathbb{R}}^3}(\phi_{u_n}^{t}u_n- \phi_u^{t} u)(u_n-u)\,dx\\
 & \quad -\int_{{\mathbb{R}}^3}(f(x,u_n)-f(x,u))(u_n-u)\,dx.
\endaligned
\end{equation*}
It is clear that
$$
  \langle I'(u_n)-I'(u), u_n-u\rangle \to 0\,\,\,\text{as}\,\,\,n\to \infty.
$$
According to $(f_1)$ and the H\"older inequality, we get
\begin{align*}
    & \int_{{\mathbb{R}}^3}(f(x,u_n)-f(x,u))(u_n-u)\,dx\\
    & \leq \int_{{\mathbb{R}}^3} \Big[ \frac{a_1}{2}(|u_n|+|u|) + \frac{a_1}{p} \Big( |u_n|^{p-1}+|u|^{p-1} \Big) \Big] |u_n-u| \,dx\\
    & \leq \frac{a_1}{2} \Big(\|u_n\|_{L^2}^2+\|u\|_{L^2}^2\Big)\|u_n-u\|_{L^2}^2 +\frac{a_1}{p} \Big(\|u_n\|_{L^p}^{p-1}+\|u\|_{L^p}^{p-1}\Big)\|u_n-u\|_{L^p}.
\end{align*}
Since $u_n\to u$ strongly in $L^p({\mathbb{R}}^3)$ for any $p\in [2, 2^*_{s})$, we have
$$
  \int_{{\mathbb{R}}^3}(f(x,u_n)-f(x,u))(u_n-u)\,dx\to 0\quad
 \text{as } \,\, n\to\infty.
$$
Moreover, by the H\"older inequality, Sobolev inequality and \eqref{eq2-2}, we have
\begin{equation*}
\aligned
&\big|\int_{{\mathbb{R}}^3}\phi_{u_n}^{t}u_n(u_n-u)\,\mathrm{d}x \big|\\
&\leq \Big(\int_{\mathbb{R}^{3}}|\phi_{u_n}^{t}|^{2^*_{t}}dx\Big)^{\frac{1}{2^*_{t}}}
\Big(\int_{\mathbb{R}^{3}}|{u_n}|^{\frac{12}{3+2t}}dx\Big)^{\frac{3+2t}{12}}
\Big(\int_{\mathbb{R}^{3}}|u_n-u|^{\frac{12}{3+2t}}dx\Big)^{\frac{3+2t}{12}}\\
&\leq S_{t}C_{\frac{12}{3+2t}}\|\phi_{u_n}\|_{D^{t,2}}\|{u_n}\|_{E}\|u_n-u\|_{E}\\
&\leq c_{6}\|{u_n}\|_{E}^{3}\|u_n-u\|_{E},\\
\endaligned\end{equation*}
where $c_{6}>0$ is a constant. Again using $u_n\rightharpoonup u$ in
$E$ and $\{u_{n}\}$ is bounded in $E$, we have
$$
  \int_{{\mathbb{R}}^3}\phi_{u_n}^{t}u_n(u_n-u)\,dx\to 0\quad
\text{as }\,\, n\to\infty.
$$
Similarly, we  obtain
$$
  \int_{{\mathbb{R}}^3}\phi_{u}^{t}u(u_n-u)\,dx\to 0 \quad
\text{as }\,\, n\to\infty.
$$
Thus,
$$
  \int_{{\mathbb{R}}^3}(\phi_{u_n}^{t}u_n-\phi_u^{t}u)(u_n-u)\,dx\to 0\quad
 \text{as } \,\, n\to\infty,
$$
so that $\|u_n-u\|_E\to 0$. Therefore, we prove that $I$ satisfies $(C)_{c}$ condition for any $c\in \mathbb{R}$.
\end{proof}

\begin{lemma}\label{lem3-2}
Suppose that $(V)$, $(f_1)$, $(f_2)$ and $(f_4)$ hold. Then the functional $I$
satisfies the $(C)_{c}$ condition for all $c\in \mathbb{R}$ .
\end{lemma}

\begin{proof}
Like in the proof of Lemma \ref{lem3-1}, it suffices to consider the case $w\neq 0$
and $w=0$, the $(C)_{c}$ sequence $\{u_{n}\}$ is bounded in $E$.

If $w\neq 0$, the proof is identical to that of Lemma \ref{lem3-1}.

If $w= 0$, inspired by \cite{lj1999}, we choose a sequence
 $\{\eta_{n}\}\subset\mathbb{R}$ such that
$$
  I(\eta_{n}u_{n})=\max_{\eta\in[0,1]}I(\eta u_{n}).
$$
Fix any $m >0$, letting $v_{n}=\sqrt{4m}\, w_{n}$, one has
\begin{equation}\label{eq2-9-1}
\aligned
&v_{n}\to 0\quad \text{in}\,\,\,L^{p}(\mathbb{R}^{3}),\,\,1\leq p<2^*_{s},\\
&v_{n}\to 0\quad \text{a.e.}\,\,\,x\in \mathbb{R}^{3}.\\
\endaligned
\end{equation}
Then, by \eqref{eq2-6*}, \eqref{eq2-9-1} and Lebesgue dominated convergence theorem,
$$
  \lim_{n\to\infty}\int_{\mathbb{R}^3}F(x,v_{n})\,dx
\leq \lim_{n\to\infty}\Big(\frac{ \varepsilon}{2}
\int_{\mathbb{R}^3}|v_{n}|^{2} \,dx
+ \frac{c(\varepsilon)}{p} \int_{\mathbb{R}^3}|v_{n}|^{p} \,dx \Big)=0,
$$
So, for $n$ sufficiently large, we obtain
\begin{equation*}
I(\eta_nu_n)\geq  I(v_{n})=  2m + \frac{1}{4} \int_{\mathbb{R}^3} \phi_{v_{n}}^{t}
v_{n}^{2} \,dx
-\int_{\mathbb{R}^3}F(x,v_{n})\,dx
 \geq  2m,
\end{equation*}
which implies that $\liminf_{n\to \infty}I(\eta_{n}u_{n})\geq 2m$. By the arbitrariness of $m$, we have
\begin{equation*}\label{infty}
\lim\limits_{n\to\infty}I(\eta_{n}u_{n})=+\infty.
\end{equation*}
Since $I(0)=0$ and $I(u_{n})\to c$ as $n\to \infty$, $I(\eta u_{n})$ attains maximum at $\eta_{n}\in (0\,,\,1)$,
Thus, $\langle I'(\eta_{n}u_{n})\,,\,\eta_{n}u_{n}\rangle=o(1)$ for large $n$.
Therefore, using ($f_4$),
\begin{equation*}\aligned
   I(u_n)&-\frac{1}{4}\langle I'(u_n),u_{n}\rangle
     =\frac{1}{4} \| u_{n}\|_{E}^{2} + \frac{1}{4} \int_{\mathbb{R}^3} \Big( f(x,u_{n}) u_{n} -4F(x, u_{n})\Big) dx\\
    &=\frac{1}{4} \| u_{n}\|_{E}^{2} + \frac{1}{4} \int_{\mathbb{R}^3} \mathcal{F}(x, u_{n}) \,dx \\
    &\geq \frac{1}{4 \theta} \| \eta_{n}u_{n}\|_{E}^{2} + \frac{1}{4 \theta} \int_{\mathbb{R}^3} \mathcal{F}(x, \eta_{n}u_{n}) \,dx \\
    & = \frac{1}{\theta}\Big[\frac{1}{4} \| \eta_{n}u_{n}\|_{E}^{2} +
    \frac{1}{4} \int_{\mathbb{R}^3} \Big( f(x,\eta_{n}u_{n}) \eta_{n}u_{n} -4F(x, \eta_{n}u_{n})\Big) dx\Big]\\
    & = \frac{1}{\theta}\Big(I(\eta_{n}u_n)-\frac{1}{4}\langle I'(\eta_{n}u_n),\eta_{n}u_{n}\rangle\Big)\\
    & \to + \infty\,\,\,\text{as}\,\,n\to \infty.
 \endaligned \end{equation*}
This contradicts \eqref{eq2-13}. In any case, we deduce that the $(C)_c$ sequence $\{u_{n}\}$ is bounded in $E$.
This completes the proof.
\end{proof}

\begin{proof}[Proof of Theorem \ref{th1-1}]
For the Hilbert space $E$, we choose an orthogonal basis $\{e_{i}\}$ of $E$,
let $X_{i}=\text{span}\{e_{i}\}$, $i=1,2,\cdot\cdot\cdot$,
and define

$$
Y_{k}=\bigoplus\limits_{i=1}^{k}X_{i},\quad \quad Z_{k}=\overline{\bigoplus_{i=k+1}^{\infty}X_{i}}.
$$
Then $E=Y_{k}\bigoplus Z_{k}$. According to Lemma \ref{lem3-1} and the oddness of $f$, we know that
$I$ satisfies the $(C)_{c}$ condition for any $c\in \mathbb{R}$ and $I(-u)=I(u)$.
It remains to verify the conditions (i) and (ii) of Fountain Theorem \ref{ft}.

Verification of (i).  Since on the finite dimensional space $Y_{k}$
all norms are equivalent, there exists $C_{k}>0$ such that
\begin{equation}\label{eq2-21}
C_{k}\|u\|_{L^{p}}\geq \|u\|_{E},\quad \forall u\in Y_{k}.
\end{equation}
From ($f_2$) we deduce that, there exist $L>0$ and $M_{k}>0$, such that for all
 $x \in \mathbb{R}^3$, $|u| \geq L$, we have
 $\frac{F(x,u)}{|u|^{4}}>M_{k}$, that is,
\begin{equation}\label{eq2-22}
  F(x,u) \geq M_{k} |u|^{4},\quad \forall x \in \mathbb{R}^3,\,\,\forall |u| \geq L.
\end{equation}

By $(f_1)$, one has
$$
|F(x,u)|\leq \Big(\frac{a_{1}}{2}+\frac{a_{1}}{p}L^{p-2}\Big)|u|^{2},\quad \forall x\in \mathbb{R}^{3},\,\,\,\forall |u|\leq L,
$$
which and \eqref{eq2-22} implies that
\begin{equation}\label{eq2-23}
F(x,u)\geq M_{k}|u|^{4}-c_{7}|u|^{2},\quad \forall (x,u)\in \mathbb{R}^{3}\times\mathbb{R},
\end{equation}
where $0<c_{7}<\frac{a_{1}}{2}+\frac{a_{1}}{p}L^{p-2}$.
Combining \eqref{eq2-21}, \eqref{eq2-23} with \eqref{eq2-3*}, we obtain
\begin{equation}\label{eq2-24}
\aligned
I (u)
&=\frac{1}{2}\|u\|_{E}^{2}+\frac{1}{4}\int_{\mathbb{R}^{3}}\phi_{u}^{t}u^{2}dx-\int_{\mathbb{R}^{3}}F(x,u)dx\\
&\leq \frac{1}{2} \|u\|_E^2 + \frac{\widetilde{C}_{2}} {4} \|u\|_E^4
- M_{k} \int_{\mathbb{R}^{3}}|u|^{4}dx + c_{7}\int_{\mathbb{R}^{3}}|u|^{2}dx \\
&\leq \frac{1}{2} \|u\|_E^2 + \frac{1} {4}\Big(\widetilde{C}_{2}-4\frac{M_{k}}{C_{k}^{4}}\Big) \|u\|_E^4
+ c_{7} \|u\|_{E}^2,\\
\endaligned
\end{equation}
for all $u\in Y_{k}$. Let $M_{k}$ large enough such that
$\widetilde{C}_{2}-4\frac{M_{k}}{C^{4}_{k}}< 0$, and  choosing
$$
\rho_{k}\geq \max\{\Big(\frac{C_{k}^{4}(2+4c_{7})}{4M_{k}-\widetilde{C}_{2}C_{k}^{4}}     \Big)^{\frac{1}{2}}\,\,,\,\,1\},
$$
inequality \eqref{eq2-24} implies that
$$
  a_k:=\max _{u\in Y_k, \|u\|_E=\rho_k}I(u)<0
$$
for some $\rho_k >0$ large enough.

Verification of (ii).
  For any $2\leq p<2^*_{s}$, taking
\begin{equation}\label{eq2-25}
    \beta_k:=\sup_{u\in Z_k, \|u\|_E=1}\|u\|_{L^p},
\end{equation}
one has $\beta_{k}\to 0$ as $k\to \infty$(see {\cite[Lemma 2.5]{ct2009}}).
From $(f_1)$, for any $\varepsilon>0$ there exists $C_{\varepsilon}>0$ such that
$$
|F(x,u)|\leq \varepsilon |u|^{2}+C_{\varepsilon}|u|^{p}.
$$
Moreover, due to $\phi_{u}^{t}>0$ for all $u\in H^{s}(\mathbb{R}^{3})$, we have
\begin{equation*}
\aligned
  I(u)
  &=\frac{1}{2}\|u\|_{E}^{2}+\frac{1}{4}\int_{\mathbb{R}^{3}}\phi_{u}^{t}u^{2}dx-\int_{\mathbb{R}^{3}} F(x,u)dx\\
  &\geq \frac{1}{2}\|u\|_{E}^{2}-\int_{\mathbb{R}^{3}} |F(x,u)|dx\\
  &\geq  \frac{1}{2}\|u\|_E^2 -\varepsilon \int_{\mathbb{R}^{3}}|u|^{2}dx-C_{\varepsilon}\int_{\mathbb{R}^{3}}|u|^pdx\\
      &\geq  \Big( \frac{1}{2} - \frac{\varepsilon}{V_0} \Big)\|u\|_E^2 - C_{\varepsilon} {\beta_k}^p \|u\|_E^p,
\endaligned
\end{equation*}
where $V_0$ is a lower bound of $V(x)$ from $(V)$ and $\beta _k$ are
defined in \eqref{eq2-25}. Choosing $r_k:=(V_0 p\beta_k^p)^{\frac{1}{2-p}}$, we
obtain
\begin{equation*}
\aligned
    b_k &=  \inf _{u\in Z_k, \|u\|_E=r_k} I(u)\\
     &\geq  \inf _{u\in Z_k, \|u\|_E = r_k}
\Big[ \Big( \frac{1}{2} - \frac{\varepsilon}{V_1} \Big)
 \|u\|_E^2 - C_{\varepsilon} {\beta_k}^p \|u\|_E^p \Big]\\
     &\geq  \Big(\frac{1}{2}-\frac{\varepsilon}{V_1}- \frac{1}{p} \Big) (C_{\varepsilon} p \beta_k^p )^{\frac{2}{2-p}}.
  \endaligned
  \end{equation*}
Because $\beta_k\to 0$ as $k\to 0$ and $p>2$, we have
$$
  b_k\geq \Big(\frac{1}{2}-\frac{\varepsilon}{V_1} - \frac{1}{p} \Big)
(C_{\varepsilon} p \beta_k^p )^{\frac{2}{2-p}} \to +\infty
$$
for enough small $\varepsilon$. This proves (ii).
Now, by Theorem \ref{ft}, $I$  possesses a sequence of critical points $\{u_{k}\}\subset E$
such that $I(u_{k})\to +\infty$ as $k\to \infty$. This complete the proof of Theorem \ref{th1-1}.
\end{proof}

\begin{proof}[Proof of Theorem \ref{th1-2}]
By virtue of Lemma \ref{lem3-2} and assumption $(f_{5})$, we see that $I$ satisfies the $(C)_{c}$
condition and is even in $u$. Like in the proof of Theorem \ref{th1-1},
assumptions $(f_2)$ and $(f_4)$ indicate that $I$ satisfies the conditions (i) and (ii) of Theorem \ref{ft}.
Hence Theorem \ref{th1-2} holds.
\end{proof}

\end{document}